\documentclass[11pt]{article}
	\usepackage[hmargin = 1.0in, vmargin = 1.0in]{geometry}
\usepackage{graphicx}
\usepackage{amssymb}
\usepackage{mathtools}
\usepackage{amsmath}
\usepackage{amsthm}
\usepackage{esdiff}
\usepackage{cite}
\usepackage{booktabs}

\DeclareMathOperator{\soft}{soft}
\DeclareMathOperator{\firm}{firm}
\DeclareMathOperator{\sign}{sign}

 \newcommand{\diag}{\mathrm{diag}}

\DeclarePairedDelimiter{\norm}{\lVert}{\rVert}		

\DeclarePairedDelimiter{\abs}{\lvert}{\rvert}		

\newcommand{\inv}{^{-1}}
\newcommand{\myJ}{ \mathrm{j} }

\newcommand{\lam}{{\lambda} }
\newcommand{\half}{\frac{1}{2}}

\newcommand{\iter}[1]{ ^{(#1)} }

\newcommand{\tr}{^{\mathsf{T}}}			
\newcommand{\ct}{^{\mathsf{H}}}			

\newcommand{\RR}{\mathbb{R}}

\newcommand{\CC}{\mathbb{C}}

\renewcommand{\leq}{\leqslant}

\renewcommand{\le}{\leqslant}
\renewcommand{\ge}{\geqslant}

\newcommand{\mge}{\succcurlyeq}		
\newcommand{\mle}{\preccurlyeq}		

\newcommand{\ia}{({\it i\/})}
\newcommand{\ib}{({\it ii\/})}

\newcommand{\infconv}{\mbox{\footnotesize$\;\square\;$}}

\newcommand{\nulls}{\operatorname{null}}
\newcommand{\ranges}{\operatorname{range}}
\newcommand{\compose}{\mathbin{\circ}}
\newcommand{\adot}{ {\,\cdot\,} }    			
\newcommand{\RX}{ \RR \cup\{ \pinf \}  }

\newcommand{\thalf}{\tfrac{1}{2}}
\newcommand{\me}{^{\mathsf{M}}}			
\newcommand{\pinv}{^{\mathsf{+}}}			
\newcommand{\pitr}{^{\mathsf{+T}}}			

\newcommand{\pinf}{\ensuremath{+\infty}}
\providecommand{\Gz}{\Gamma_0}

\newcommand{\opt}{^{\mathsf{opt}}}			

\newtheorem{defn}{Definition}
\newtheorem{prop}{Proposition}
\newtheorem{lemma}{Lemma}
\newtheorem{theorem}{Theorem}
\newtheorem{corr}{Corollary}
\newtheorem{example}{Example}

\title{Sparse Regularization via Convex Analysis}
\renewcommand\footnotemark{}

\author{Ivan Selesnick%
	\thanks{
	Department of Electrical and Computer Engineering, 
	Tandon School of Engineering, 
	New York University, 
	New York, 
	USA.
	Email: selesi@nyu.edu.
	This work was supported by
	NSF under grant~CCF-1525398
	and
	ONR under grant~N00014-15-1-2314.
	Software is available online:
	\newline
	\hspace*{2em}
	{https://codeocean.com/2017/06/21/gmc-sparse-regularization}  (Matlab)
	\newline
	\hspace*{2em}
	{https://codeocean.com/2017/06/22/gmc-sparse-regularization}  (Python)
	\newline
	\hspace*{2em}
	{http://ieeexplore.ieee.org/document/7938377/media}  (Matlab)
	\newline
	\newline
I. Selesnick, ``Sparse Regularization via Convex Analysis'' in \emph{IEEE Transactions on Signal Processing}, vol.~65, no.~17, pp. 4481-4494, Sept. 2017.
doi: 10.1109/TSP.2017.2711501	}%
}

\date{}

\begin{document}
\maketitle

\begin{abstract}

Sparse approximate solutions to linear equations are classically obtained
via L1 norm regularized least squares, 
but this method often underestimates the true solution. 
As an alternative to the L1 norm, 
this paper proposes a class of non-convex penalty functions
that maintain the convexity of the least squares cost function to be minimized,
and avoids the systematic underestimation characteristic of L1 norm regularization. 
The proposed penalty function is a multivariate generalization of the minimax-concave (MC) penalty.
It is defined in terms of a new multivariate generalization of the Huber function,
which in turn is defined
via
  infimal convolution. 
The proposed sparse-regularized least squares cost function 
can be minimized by proximal algorithms comprising simple computations. 
\end{abstract}

\section{Introduction}

Numerous signal and image processing techniques 
build upon sparse approximation \cite{Starck_2015_book}.
A sparse approximate solution to a system of linear equations ($ y = A x $)
can often be obtained via convex optimization.
The usual technique is to minimize the regularized linear least squares cost function
$ J \colon \RR^N \to \RR $,
\begin{equation}
	\label{eq:bpd}
	J(x) = \half \norm{ y - A x }_2^2 + \lam \norm{ x }_1,
	\quad \lam > 0.
\end{equation}
The $ \ell_1 $ norm is classically used as a regularizer here, since
among convex regularizers it induces sparsity most effectively \cite{Bruckstein_2009_SIAM}.
But this formulation tends to underestimate high-amplitude components 
of $ x \in \RR^N $.
Non-convex sparsity-inducing regularizers are also widely used
(leading to more accurate estimation of high-amplitude components),
but then the cost function is generally non-convex
and has extraneous suboptimal local minimizers \cite{Nikolova_2011_chap}. 

This paper proposes a class of non-convex penalties
for sparse-regularized linear least squares
that generalizes the $ \ell_1 $ norm
and maintains the convexity of the least squares cost function to be minimized.
That is, we consider the cost function $ F \colon \RR^N \to \RR $ 
\begin{equation}
	F(x) = \half \norm{ y - A x }_2^2 + \lam \, \psi_B(x), 
	\quad
	\lam > 0
\end{equation}
and we propose a new non-convex penalty $ \psi_B \colon \RR^N \to \RR $ 
that makes $ F $ convex.  
The penalty $ \psi_B $ is parameterized by a matrix $ B $, 
and 
the convexity of $ F $ depends on $ B $ being suitably prescribed. 
In fact, the choice of $ B $ will depend on $ A $.

The matrix (linear operator) $ A $ may be arbitrary (i.e., injective, surjective, both, or neither).
In contrast to the $ \ell_1 $ norm, the new approach does not systematically underestimate
high-amplitude components of sparse vectors.
Since the proposed formulation is convex, the
cost function has no suboptimal local minimizers.

The new class of non-convex penalties is defined using tools of convex analysis.
In particular, \emph{infimal convolution} is used to define a new multivariate generalization of the Huber function. 
In turn, the generalized Huber function is used to define the proposed non-convex penalty,
which can be considered a multivariate generalization of the minimax-concave (MC) penalty.
Even though the generalized MC (GMC) penalty is non-convex,
it is easy to prescribe this penalty so as to maintain the convexity of the cost function to be minimized.

The proposed convex cost functions 
can be minimized using proximal algorithms, comprising simple computations.
In particular, the minimization problem can be cast as a kind of saddle-point problem
for which the forward-backward splitting algorithm is applicable. 
The main computational steps of the algorithm are the operators $ A $, $ A\tr $, and soft thresholding.
The implementation is thus `matrix-free' in that it involves the operators $ A $ and $ A\tr $,
but does not access or modify the entries of $ A $.
Hence, 
the algorithm
can leverage efficient implementations of $ A $ and its transpose.

We remark that while the proposed GMC penalty is non-separable,
we do not advocate non-separability in and of itself as a desirable property of a sparsity-inducing penalty. 
But in many cases (depending on $ A $),
non-separability is simply a requirement of a non-convex penalty designed so as to maintain convexity of the cost function $ F $ to be minimized.
If $ A\tr \! A $ is singular
(and none of its eigenvectors are standard basis vectors), then a separable penalty that maintains the convexity of the cost function $ F $ must, in fact, be a convex penalty \cite{Selesnick_2016_TSP_BISR}. 
This leads us back to the $ \ell_1 $ norm.
Thus, to improve upon the $ \ell_1 $ norm, the penalty must be non-separable.

This paper is organized as follows.
Section \ref{sec:notation} sets notation and recalls definitions of convex analysis. 
Section \ref{sec:scalar} 
recalls the (scalar) Huber function, the (scalar) MC penalty, and how they arise
in the formulation of threshold functions (instances of proximity operators).
The subsequent sections generalize these concepts to the multivariate case.
In Section \ref{sec:huber}, we define a multivariate version of the Huber function.
In Section \ref{sec:gmc}, we define a multivariate version of the MC penalty.
In Section \ref{sec:spreg}, we show how to set the GMC penalty to maintain
convexity of the least squares cost function.
Section \ref{sec:alg} presents a proximal algorithm to minimize this
type of cost function.
Section \ref{sec:examples} presents examples wherein the GMC penalty is used
for signal denoising and approximation.

Elements of this work were presented in Ref.~\cite{Selesnick_2017_ICASSP_GMC}.

\subsection{Related work}
\label{sec:prior}

Many prior works have proposed non-convex penalties that strongly promote sparsity 
or describe 
algorithms for solving the sparse-regularized linear least squares problem,
e.g., \cite{Castella_2015_camsap, Candes_2008_JFAP, Nikolova_2011_chap, Mohimani_2009_TSP, Marnissi_2013_ICIP_shortercite, Chartrand_2014_ICASSP, Fan_2001_JASA, Gholami_2011_TSP, Chouzenoux_2013_SIAM, Portilla_2007_SPIE, Zou_2008_AS, Chen_2014_TSP_Convergence, Gasso_2009_TSP,Soussen_2015_TSP, Woodworth_2016_InvProb}.
However, most of these papers
\ia~use separable (additive) penalties
or
\ib~do not seek to maintain convexity of the cost function.
Non-separable non-convex penalties are proposed in Refs.~\cite{Tipping_2001_JMLR, Wipf_2011_tinfo},
but they are not designed to maintain cost function convexity.
The development of convexity-preserving non-convex penalties
was pioneered by Blake, Zisserman, and Nikolova \cite{Blake_1987, Nikolova_1998_ICIP, Nikolova_2010_TIP, Nikolova_2011_chap, Nikolova_1999_TIP}, 
and further developed in 
\cite{Bayram_2016_TSP, Chen_2014_TSP_ncogs, Ding_2015_SPL, He_2016_MSSP, Lanza_2016_JMIV, MalekMohammadi_2016_TSP, Parekh_2016_SPL_ELMA, Selesnick_2014_TSP_MSC, Selesnick_SPL_2015}.
But these are separable penalties, and as such they are fundamentally limited. 
Specifically, 
if $ A\tr \! A $ is singular, 
then a separable penalty constrained to maintain cost function convexity can only improve on the $ \ell_1 $ norm to a very limited extent \cite{Selesnick_2016_TSP_BISR}.
Non-convex regularization that maintains cost function convexity was used in \cite{Lanza_2016_NM}
in an iterative manner for non-convex optimization,
to reduce the likelihood that an algorithm converges to suboptimal local minima.

To overcome the fundamental limitation of separable non-convex penalties,
we proposed a bivariate non-separable non-convex penalty 
that maintains the convexity of the cost function to be minimized \cite{Selesnick_2016_TSP_BISR}.
But that penalty is useful for only a narrow class of linear inverse problems.
To handle more general problems, 
we subsequently proposed a multivariate penalty formed by subtracting from the $ \ell_1 $ norm
a function comprising the composition of a linear operator and a separable nonlinear function \cite{Selesnick_2017_TSP_MUSR}.
Technically, this type of multivariate penalty is non-separable, but it still
constitutes a rather narrow class of non-separable functions.

Convex analysis tools (especially the Moreau envelope and the Fenchel conjugate) 
have recently been used 
in novel ways 
for sparse regularized least squares
 \cite{Carlsson_2016_arxiv, Soubies_2015_SIAM_CEL0}.
Among other aims, these papers seek the convex envelope of the $ \ell_0 $ pseudo-norm
regularized least squares cost function,
and derive alternate cost functions that share the same global minimizers but have fewer local minima. 
In these approaches,
algorithms are less likely to converge to suboptimal local minima
(the global minimizer might still be difficult to calculate). 

For the special case where $ A\tr \! A $ is diagonal,
the proposed GMC penalty is closely related to
the `continuous exact $ \ell_0 $' (CEL0) penalty
introduced in \cite{Soubies_2015_SIAM_CEL0}.
In \cite{Soubies_2015_SIAM_CEL0} it is observed that 
if $ A\tr \! A $ is diagonal, then the global minimizers of the $ \ell_0 $ regularized problem 
coincides with that of a convex function 
defined using the CEL0 penalty. 
Although the diagonal case is simpler than the non-diagonal case (a non-convex penalty can be readily constructed to maintain cost function convexity \cite{Selesnick_2014_TSP_MSC}),
the connection to the $ \ell_0 $ problem is enlightening. 

In other related work,
we use convex analysis concepts (specifically, the Moreau envelope)
for the problem of total variation (TV) denoising \cite{Selesnick_2017_SPL_MTVD}.
In particular, 
we prescribe a non-convex TV penalty that preserves the convexity of
the TV denoising cost function to be minimized.
The approach of Ref.~\cite{Selesnick_2017_SPL_MTVD}
generalizes standard TV denoising so as to 
more accurately estimate jump discontinuities.

\section{Notation}
\label{sec:notation}

The $ \ell_1 $, $ \ell_2 $, and $ \ell_{\infty} $ norms of $ x \in \RR^N $ are defined
$ \norm{ x }_1 = \sum_n \abs{ x_n } $,
$ \norm{ x }_2 = \bigl( \sum_n \abs{ x_n }^2 \bigr)^{1/2} $,
and
$ \norm{ x }_\infty = \max_n \abs{ x_n } $.
If $ A \in \RR^{M \times N} $, then component $ n $ of $ A x $ is denoted $ [A x]_n $.
If the matrix $ A - B $ is positive semidefinite, we write $ B \mle A $. 
The matrix 2-norm of matrix $ A $ is denoted $ \norm{ A }_2 $
and its value is the square root of the maximum eigenvalue of $ A\tr\! A $.
We have $ \norm{ A x }_2 \le \norm{A }_2 \norm{ x }_2 $ for all $ x \in \RR^N $.
If $ A $ has full row-rank (i.e., $ A A\tr $ is invertible), 
then the pseudo-inverse of $ A $ is given by $ A\pinv := A\tr (A A\tr)\inv $.
We denote the transpose of the pseudo-inverse of $ A $ as $ A\pitr $, i.e., $ A\pitr := ( A\pinv )\tr $.
If $ A $ has full row-rank, then $ A\pitr = (A A\tr)\inv A $.

This work uses definitions and notation of convex analysis \cite{Bauschke_2011}.
The infimal convolution of two functions $ f $ and $ g $ from $ \RR^N $ to $ \RX $ is given by
\begin{equation}
	(f \infconv g)(x)
	=
	\inf_{ v \in \RR^N }
	\bigl\{
		f( v ) + g( x - v )
	\big\}.
\end{equation}
The Moreau envelope of the function $ f \colon \RR^N \to \RR $ is given by 
\begin{equation}
	f\me(x) = \inf_{ v \in \RR^N }
	\bigl\{
		f(v) + \thalf \norm{ x - v }_2^2 
	\bigr\}.
\end{equation}
In the notation of infimal convolution, we have
\begin{equation}
	f\me = f \infconv \thalf \norm{ \adot }_2^2 .
\end{equation}
The set of proper lower semicontinuous (lsc) convex functions from $ \RR^N $ to $ \RX $ is  denoted $ \Gz(\RR^N) $.

If the function $ f $ is defined as the composition $ f(x) = h(g(x)) $, then we write $ f = h \compose g $.

The soft threshold function $ \soft \colon \RR \to \RR $ 
with threshold parameter $ \lam \ge 0 $ is defined as
\begin{equation}
	\soft(y; \lam) :=
	\\
	\begin{cases}
		0, \quad & \abs{ y } \le \lam
		\\
		( \abs{ y } - \lam ) \sign(y), \quad &  \abs{ y } \ge  \lam.
	\end{cases}
\end{equation}

\begin{figure}[t]
	\centering
	\includegraphics{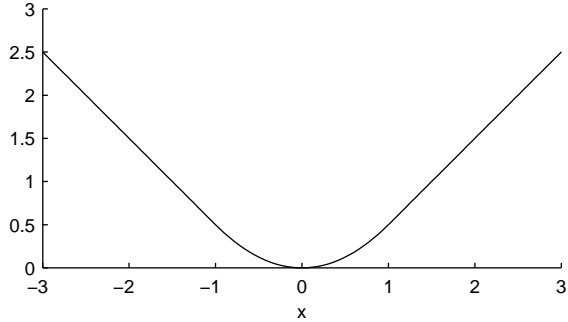}
	\caption{
		The Huber function.
	}
	\label{fig:scalar_huber}
\end{figure}

\section{Scalar Penalties}
\label{sec:scalar}

We recall the definition of the Huber function \cite{Huber_1964}.

\begin{defn}
\label{def:huber}
The Huber function
$ s \colon \RR \to \RR $
is defined as
\begin{equation}
	\label{eq:defss}
	s(x) := 
	\begin{cases}
		\thalf x^2, \quad & \abs{ x } \le 1 
		\\
		\abs{x} - \half, & \abs{ x } \ge 1,
	\end{cases}
\end{equation}
as illustrated in Fig.~\ref{fig:scalar_huber}.
\end{defn}

\begin{prop}
The Huber function can be written as
\begin{equation}
	\label{eq:sica}
	s(x) = \min_{ v \in \RR } \{ \abs{ v } + \thalf (x - v)^2 \}.
\end{equation}
In the notation of infimal convolution, we have equivalently
\begin{equation}
	\label{eq:sic}
	s = \abs{ \adot } \infconv \thalf ( \adot )^2 .
\end{equation}
And in the notation of the Moreau envelope, we have equivalently
$ s = \abs{ \adot } \me $.
\end{prop}

The Huber function is a standard example of the Moreau envelope.
For example,
see
Sec.~3.1 of Ref.~\cite{Parikh_2014_FTO}
and
\cite{Combettes_2017_SVVA}.
We note here that, given $ x \in \RR $, 
the minimum in \eqref{eq:sica} is achieved for $ v $ equal to $ 0 $, $ x - 1 $, or $ x + 1 $, 
i.e.,
\begin{equation}
	s(x) = \min_{ v \in \{ 0 , \, x - 1,  \, x + 1 \} }
	\{ 
		\abs{ v } + \thalf (x - v)^2
	\}.
\end{equation}
Consequently, the Huber function can be expressed as
\begin{equation}
	\label{eq:smmm}
	s(x) = \min \bigl\{
		\thalf x^2, \ \abs{ x - 1} + \thalf, \ \abs{ x + 1} + \thalf
	\bigr\}
\end{equation}
as illustrated in Fig.~\ref{fig:scalar_huber_min}.

\begin{figure}[t]
	\centering
	\includegraphics{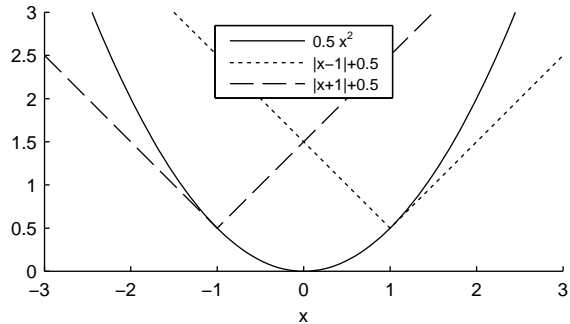}
	\caption{
		The Huber function as the pointwise minimum of three functions.
	}
	\label{fig:scalar_huber_min}
\end{figure}

\bigskip

\noindent
We now consider the scalar penalty function illustrated in Fig.~\ref{fig:mcpenalty}.
This is the minimax-concave (MC) penalty \cite{Zhang_2010_AnnalsStat};
see also \cite{Fornasier_2008_ACHA, Bayram_2016_TSP, Bayram_2015_SPL}.

 \begin{defn}
\label{def:mcpen}
The minimax-concave (MC) penalty function
$ \phi \colon \RR \to \RR $
is defined as
\begin{equation}
	\label{eq:defMC}
	\phi( x ) := 
	\begin{cases}
		\abs{ x } - \frac{ 1 }{ 2 } x^2, \  & \abs{ x } \le 1
		\\
		\frac{ 1 }{ 2 }, & \abs{ x } \ge 1,
	\end{cases}	
\end{equation}
as illustrated in Fig.~\ref{fig:mcpenalty}.
\end{defn}
The MC penalty can be expressed as
\begin{equation}
	\label{eq:absms}
	\phi( x ) = \abs{ x } - s( x ) 
\end{equation}
where $ s $ is the Huber function.
This representation of the MC penalty
will be used in Sec.~\ref{sec:gmc} to generalize the MC penalty to the multivariate case.

\begin{figure}
	\centering
	\includegraphics{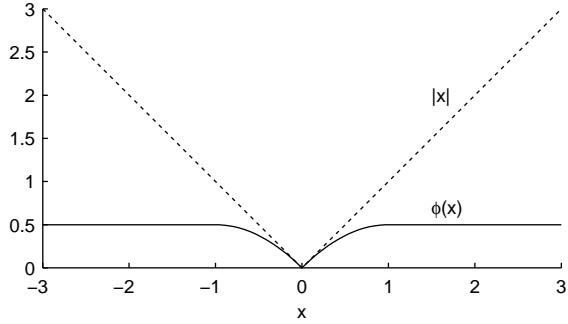}
	\caption{
		The MC penalty function.
	}
	\label{fig:mcpenalty}
\end{figure}

\subsection{Scaled functions}

It will be convenient to define scaled versions of the Huber function and MC penalty.

\begin{defn}
\label{def:scaled_huber}
Let $ b \in \RR $.
The scaled Huber function
$ s_{ b } \colon \RR \to \RR $
is defined as
\begin{equation}
	\label{eq:defsss}
	s_{ b }( x ) 
	:= s( { b }^2 x ) / { b }^2,
	\ \
	b \neq 0.
\end{equation}
For $ b = 0 $, the function is defined as
\begin{equation}
	s_0(x) := 0.
\end{equation}
\end{defn}

Hence, for $ b \neq 0 $, the scaled Huber function is given by
\begin{equation}
	s_{ b }( x ) 
	= 
	\begin{cases}
		\half { b }^2 x^2, \quad & \abs{ x } \le 1/{ b }^2 
		\\
		\abs{x} - \frac{ 1 }{ 2 { b }^2 }, & \abs{ x } \ge 1/{ b }^2.
	\end{cases}
\end{equation}
The scaled Huber function $ s_b $
is shown in Fig.~\ref{fig:scaled}
for several values of the scaling parameter $ b $.
Note that
\begin{equation}
	0 \le s_b(x) \le \abs{ x },
	\quad
	\forall x \in \RR,
\end{equation}
and
\begin{align}
	\lim_{ { b } \to \infty }
	 s_{ b }( x ) & = \abs{ x }  
	 \\
	\lim_{ { b } \to 0 }
	 s_{ b }(x) & = 0.
\end{align}
Incidentally, we use $ b^2 $ in definition \eqref{eq:defsss} rather than $ b $,
so as to parallel the generalized Huber function to be defined in Sec.~\ref{sec:huber}.

\begin{prop}
Let $ b \in \RR $.
The scaled Huber function can be written as
\begin{equation}
	\label{eq:saic}
	s_{ b }(x)
	= 
	\min_{ v \in \RR } \{ \abs{ v } + \thalf { b }^2 (x - v)^2 \} .
\end{equation}
In terms of infimal convolution, we have equivalently
\begin{equation}
	s_{ b } = \abs{ \adot } \infconv \thalf { b }^2 ( \adot )^2.
\end{equation}
\end{prop}

\begin{proof}
For $ b \neq 0 $, we have from \eqref{eq:defsss} that
\begin{align*}
	s_{ b }(x)
	& = \min_{ v \in \RR } \{ \abs{ v } + \thalf ({ b }^2 x - v)^2 \} / { b }^2
	\\
	& = \min_{ v \in \RR } \{ \abs{ { b }^2 v } + \thalf ({ b }^2 x - { b }^2 v)^2 \} / { b }^2
	\\
	& = \min_{ v \in \RR } \{ \abs{ v } + \thalf { b }^2 (x - v)^2 \} .
\end{align*}
It follows from $ \abs{ \adot } \infconv 0 = 0 $
that \eqref{eq:saic} holds for $ b = 0 $.
\end{proof}

\begin{defn}
\label{def:smcpen}
Let $ b \in \RR $.
The scaled MC penalty function
$ \phi_{ b } \colon \RR \to \RR $
is defined as
\begin{align}
	\label{eq:defsMC}
	\phi_{ b }( x )
	& := 
	\abs{ x } - s_{ b }( x )
\end{align}
where $ s_{ b } $ is the scaled Huber function.
\end{defn}
The scaled MC penalty $ \phi_b $ 
is shown in Fig.~\ref{fig:scaled}
for several values of $ b $.
Note that
$ \phi_0(x) = \abs{ x } $.
For $ b \neq 0 $,
\begin{equation}
	\phi_b(x) 
	=
	\begin{cases}
		\abs{ x } - \half { b }^2 x^2, \quad & \abs{ x } \le 1/{ b }^2 
		\\
		\frac{ 1 }{ 2 { b }^2 }, & \abs{ x } \ge 1/{ b }^2.
	\end{cases}	
\end{equation}

\begin{figure}
	\centering
	\includegraphics{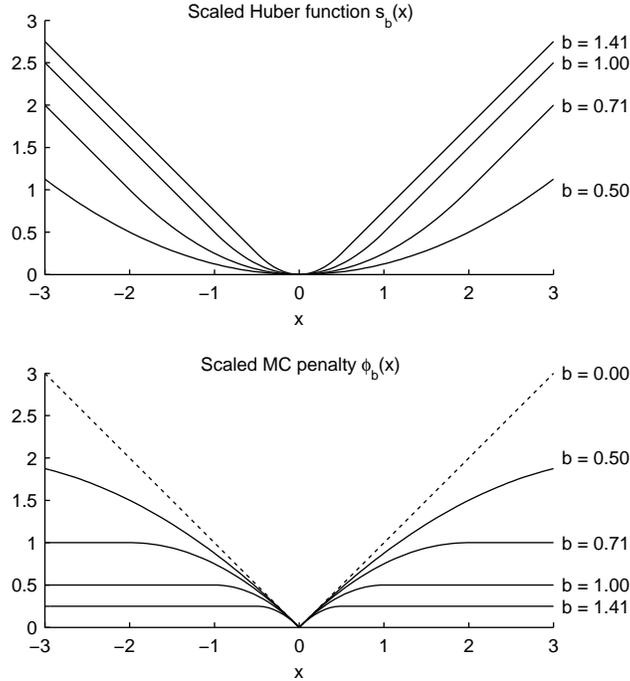}
	\caption{
		Scaled Huber function and MC penalty for several values of the scaling parameter.
	}
	\label{fig:scaled}
\end{figure}

\subsection{Convexity condition}

In the scalar case, the MC penalty corresponds to a type of threshold function.
Specifically,
the \emph{firm} threshold function is the \emph{proximity operator} of the MC penalty,
provided that a particular convexity condition is satisfied. 
Here, we give the convexity condition for the scalar case. 
We will generalize this condition to the multivariate case in Sec.~\ref{sec:spreg}.

\begin{prop}
\label{prop:sclar_convex}
Let $ \lam > 0 $ and $ a \in \RR $.
Define $ f \colon \RR \to \RR $,
\begin{equation}
	\label{eq:scprop_f}
	f(x) = \half (y - a x)^2 + \lam \, \phi_{ b }( x ) 
\end{equation}
where $ \phi_b $ is the scaled MC penalty \eqref{eq:defsMC}.
If
\begin{equation}
	\label{eq:scc}
	{ b }^2 \le a^2/\lam , 
\end{equation}
then $ f $ is convex. 
\end{prop}

There are several ways to prove Proposition \ref{prop:sclar_convex}.
In anticipation of the multivariate case, 
we use a technique in the following proof 
that we later use in the proof of Theorem~\ref{thm:cc} in Sec. \ref{sec:spreg}.

\begin{proof}[Proof of Proposition \ref{prop:sclar_convex}]
Using \eqref{eq:defsMC}, we write $ f $ as
\begin{align*}
	f(x)
	& =
	\thalf (y - a x)^2 + \lam \abs{ x } -  \lam s_{ b }( x ) 
	\\
	& =
	g(x) + \lam  \abs{ x } 
\end{align*}
where
$ s_b $ is the scaled Huber function
and
$ g \colon \RR \to \RR $ is given by 
\begin{equation}
	\label{eq:gfun}
	g(x) =
	\thalf (y - a x)^2 - \lam s_{ b }( x ).
\end{equation}
Since the sum of two convex functions is convex,
it is sufficient to show $ g $ is convex. 
Using \eqref{eq:saic}, we have
\begin{align*}
	g(x) & =
	\thalf (y - a x)^2 - \lam \min_{ v \in \RR } \{ \abs{ v } + \thalf { b }^2 (x - v)^2 \} 
	\\
	& =
	\max_{ v \in \RR }
	\left\{
		\thalf (y - a x)^2 - \lam \abs{ v } - \thalf \lam { b }^2 (x - v)^2  
	\right\}
	\\
	& =
	\thalf ( a^2 - \lam { b }^2 ) x^2 
	+
	\max_{ v \in \RR }
	\big\{
		\thalf (y^2 - 2 a x y)
		 - \lam \abs{ v } - \thalf \lam { b }^2 ( v^2  - 2 x v )
	\big\}.
\end{align*}
Note that the expression in the curly braces is affine (hence convex) in $ x $.
Since the pointwise maximum of a set of convex functions is itself convex, 
the second term is convex in $ x $.
Hence, $ g $ is convex if $ a^2 - \lam { b }^2 \ge 0 $.
\end{proof}

The firm threshold function was defined by Gao and Bruce \cite{Gao_1997}
as a generalization of hard and soft thresholding.

\begin{figure}
	\centering
	\includegraphics{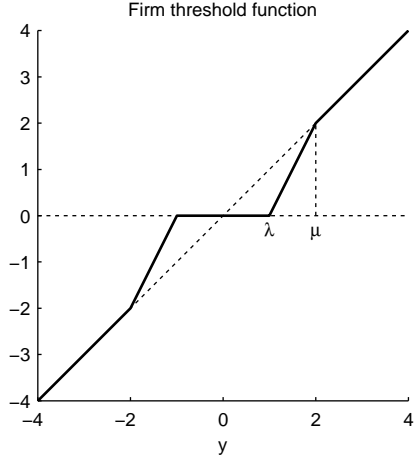}
	\caption{
		Firm threshold function.
	}
	\label{fig:firm}
\end{figure}

\begin{defn}
\label{def:firm}
Let $ \lam > 0 $ and $ \mu > \lam $.
The threshold function $ \firm \colon \RR \to \RR $ is defined as
\begin{equation}
	\firm(y; \lam, \mu) :=
	\begin{cases}
		0, \quad & \abs{ y } \le \lam
		\\
		\mu ( \abs{ y } - \lam ) / (\mu - \lam )  \sign(y), \quad &  \lam \le \abs{ y } \le \mu
		\\
		y, & \abs{ y } \ge \mu
	\end{cases}
\end{equation}
as illustrated in Fig.~\ref{fig:firm}.
\end{defn}
In contrast to the soft threshold function, the firm threshold function does not underestimate large amplitude values,
since it equals the identity for large values of its argument. 
As $ \mu \to \lam $ or $ \mu \to \infty $, the firm threshold function approaches the hard or soft threshold function, respectively.

\medskip

\noindent
We now state the correspondence between 
the MC penalty and the firm threshold function.
When $ f $ in \eqref{eq:scprop_f} is convex (i.e., $ b^2 \le a^2 / \lam $), 
the minimizer of $ f $ is given by firm thresholding. 
This is noted in Refs.~\cite{Fornasier_2008_ACHA, Zhang_2010_AnnalsStat, Woodworth_2016_InvProb, Bayram_2015_SPL}.

\begin{prop}
\label{prop:firm}
Let $ \lam > 0 $,
$ a > 0 $, $ b > 0 $,
and $ b^2 \le a^2 / \lam $.
Let $ y \in \RR $.
Then the minimizer of $ f $ in \eqref{eq:scprop_f} is given by firm thresholding, 
i.e.,
\begin{equation}
	\label{eq:sfirm}
	x\opt = \firm( y/a ; \lam/a^2 , 1/b^2 ).
\end{equation}
\end{prop}

Hence, the minimizer of the scalar function $ f $ in \eqref{eq:scprop_f} 
is easily obtained via firm thresholding.
However, 
the situation in the multivariate case is more complicated. 
The aim of this paper is to 
generalize this process to the multivariate case:
to define a multivariate MC penalty generalizing \eqref{eq:defMC},
to define a regularized least squares cost function generalizing \eqref{eq:scprop_f},
to generalize the convexity condition \eqref{eq:scc},
and 
to provide a method to calculate a minimizer. 

\section{Generalized Huber Function}
\label{sec:huber}

In this section, 
we introduce a multivariate generalization of the Huber function. 
The basic idea is to generalize \eqref{eq:saic}
which expresses the scalar Huber function as an infimal convolution.

\begin{defn}
\label{def:ghuber}
Let $ B \in \RR^{ M \times N } $.
We define the generalized Huber function $ S_B \colon \RR^N \to \RR $ as
\begin{equation}
	\label{eq:defS}
	S_B(x) :=
	\inf_{ v \in \RR^N }
	\bigl\{
		\norm{ v }_1 + \thalf \norm{ B ( x - v) }_2^2
	\bigr\}.
\end{equation}
In the notation of infimal convolution, we have
\begin{equation}
	S_B = \norm{ \adot }_1 \infconv \thalf \norm{ B \adot }_2^2 .
\end{equation}
\end{defn}

\begin{prop}
\label{prop:lsc}
The generalized Huber function $ S_B $
is a proper lower semicontinuous convex function,
 and
the infimal convolution is exact, i.e., 
\begin{equation}
	\label{eq:defSm}
	S_B(x) =
	\min_{ v \in \RR^N }
	\bigl\{
		\norm{ v }_1 + \thalf \norm{ B ( x - v) }_2^2
	\bigr\}.
\end{equation}

\end{prop}

\begin{proof}
Set $ f =  \norm{ \adot }_1 $ and $ g = \norm{ B \adot }_2^2 $.
Both $ f $ and $ g $ are convex; hence $ f \infconv g $ is convex by proposition 12.11 in \cite{Bauschke_2011}.
Since $ f $ is coercive and $ g $ is bounded below,
and
$ f, g \in \Gz(\RR^N) $,
it follows that $ f \infconv g \in \Gz(\RR^N) $
and the infimal convolution is exact
(i.e., the infimum is achieved for some $ v $)
 by Proposition 12.14 in \cite{Bauschke_2011}.
\end{proof}

Note that
if $ C\tr\! C = B\tr\! B $, then $ S_B(x) = S_C(x) $ for all $ x $.
That is, the generalized Huber function $ S_B $ depends only on $ B\tr\! B $,
not on $ B $ itself.
Therefore, without loss of generality, we may assume $ B $ has full row-rank. 
(If a given matrix $ B $ does not have full row-rank, then
there is another matrix $ C $ with full row-rank such that $ C\tr\! C = B\tr\! B $,
yielding the same function $ S_B $.)

As expected, the generalized Huber function reduces to the scalar Huber function.

\begin{prop}
If $ B $ is a scalar, i.e., $ B = b \in \RR $, 
then the generalized Huber function 
reduces to the scalar Huber function, 
$ S_b(x) = s_b(x) $ for all $ x \in \RR $.
\end{prop}

The generalized Huber function is separable (additive) when $ B\tr \! B $ is diagonal.

\begin{prop}
\label{prop:BBD}
Let $ B \in \RR^{ M \times N } $.
If $ B\tr\! B $ is diagonal, 
then the generalized Huber function is separable (additive), comprising a sum of scalar Huber functions.
Specifically,
\begin{equation}
	\nonumber
	B\tr\! B = \diag( \alpha_1^2, \dots, \alpha_N^2 )
	\ \implies \
	S_B(x) = \sum_{ n } s_{ \alpha_n }( x_n ).
\end{equation}
\end{prop}

The utility of the generalized Huber function will be most apparent when $ B\tr \! B $ is a non-diagonal matrix. 
In this case, the generalized Huber function is non-separable, as illustrated in the following two examples.

\begin{example}
For the matrix
\begin{equation}
	\label{eq:BA}
	B = 
	\begin{bmatrix}
	    1    & 0
	    \\
	   1    & 1
	   \\
	    0    & 1
	\end{bmatrix},
\end{equation}
the generalized Huber function $ S_B $ is shown in Fig.~\ref{fig:huber2_rank2_A}.
As shown in the contour plot, the level sets of $ S_B $ near the origin are ellipses.

\end{example}

\begin{example}
For the matrix
\begin{equation}
	\label{eq:BB}
	B =  [1 \ \ 0.5]
\end{equation}
the generalized Huber function $ S_B $ is shown in Fig.~\ref{fig:huber2_rank1_A}.
The level sets of $ S_B $ are parallel lines
because $ B $ is of rank 1.
\end{example}

\begin{figure}[t]
	\begin{minipage}[t]{0.45\textwidth}
		\centering
		\includegraphics{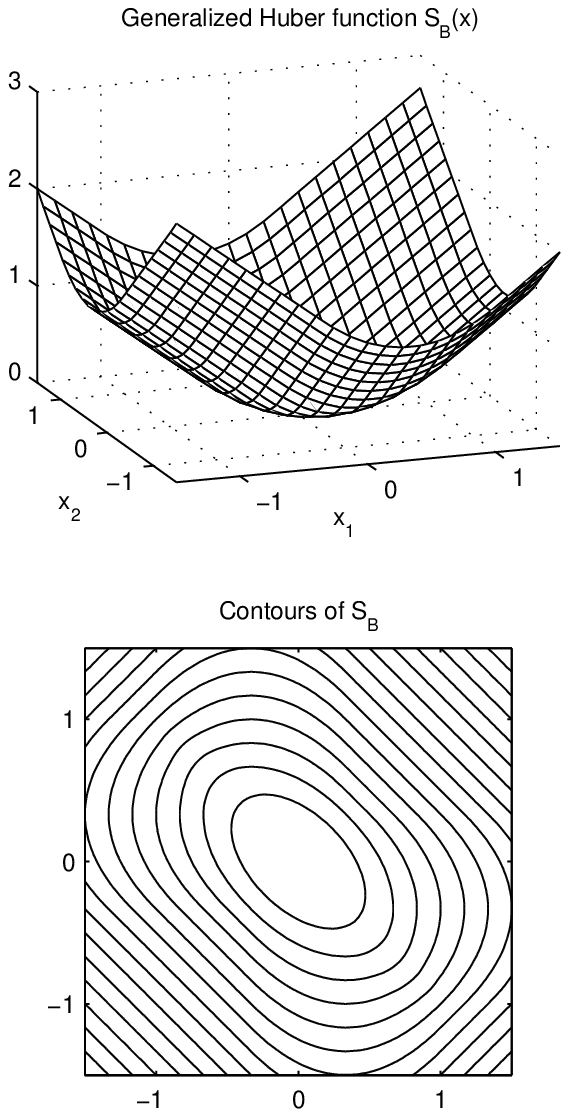}
		\caption{
			The generalized Huber function
			for the matrix $ B $ in \eqref{eq:BA}.
		}
		\label{fig:huber2_rank2_A}
	\end{minipage}
	\hfill
	\begin{minipage}[t]{0.45\textwidth}
		\centering
		\includegraphics{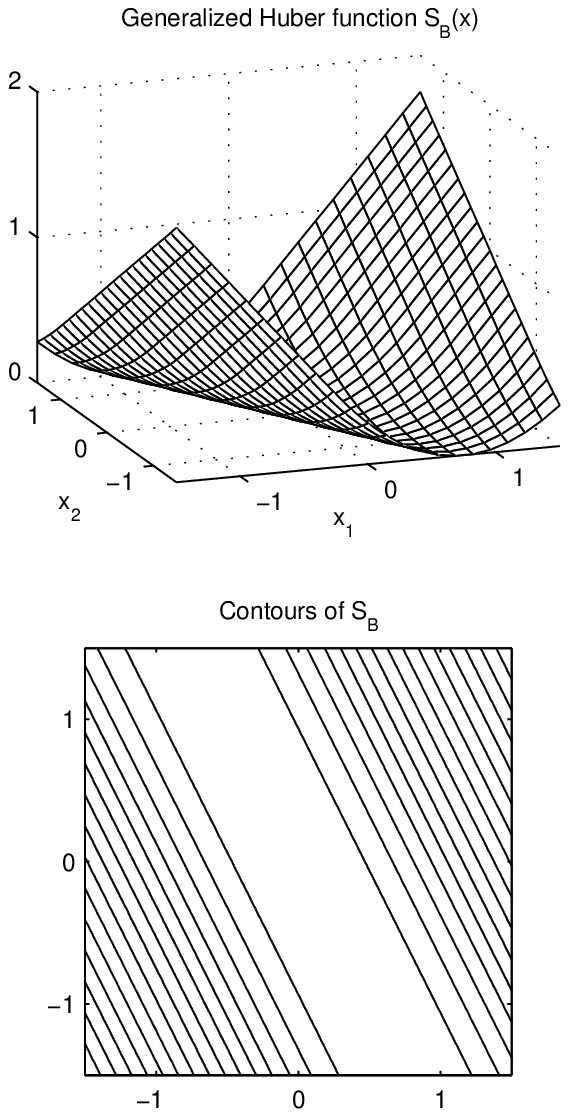}
		\caption{
			The generalized Huber function
			for the matrix $ B $ in \eqref{eq:BB}.
		}
		\label{fig:huber2_rank1_A}
	\end{minipage}
\end{figure}

There is not a simple explicit formula for the generalized Huber function. 
But, using \eqref{eq:defSm} we can derive several properties 
regarding the function.

\begin{prop}
\label{prop:SBltL1}
Let $ B \in \RR^{ M \times N } $.
The generalized Huber function satisfies
\begin{equation}
	0 \le
	S_B(x) \le \norm{ x }_1,
	\quad
	\forall x\in \RR^N.
\end{equation}
\end{prop}
\begin{proof}
Using \eqref{eq:defSm}, we have
\begin{align*}
	S_B(x) 
	& =
	\min_{ v \in \RR^N }
	\bigl\{		\norm{ v }_1 + \thalf \norm{ B ( x - v) }_2^2		\bigr\}
	\\
	& \leq
	\Bigl[
	\norm{ v }_1 + \thalf \norm{ B ( x - v) }_2^2
	\Bigr]_{ v = x }
	\\
	& = 
	\norm{ x }_1.
\end{align*}
Since $ S_B $ is the minimum of a non-negative function, 
it also follows that $ S_B(x) \ge 0 $ for all $ x $.
\end{proof}

The following proposition 
accounts for
the ellipses near the origin in the contour plot of the generalized Huber 
function in Fig.~\ref{fig:huber2_rank2_A}.
(Further from the origin, the contours are not ellipsoidal.)
\begin{prop}
\label{prop:sbnz}
Let $ B \in \RR^{ M \times N } $.
The generalized Huber function 
satisfies
\begin{equation}
	S_B(x) = \thalf \norm{ B x }_2^2
	\ \
	\text{for all} \ \ 
	\norm{ B\tr\! B x }_\infty \le 1.
\end{equation}
\end{prop}
\begin{proof}
From \eqref{eq:defSm}, 
we have that $ S_B(x) $ is the minimum value of $ g $
where $ g \colon \RR^N \to \RR $
is given by
\begin{equation*}
	g(v) = \norm{ v }_1 + \thalf \norm{ B ( x - v) }_2^2.
\end{equation*}
Note that 
\begin{equation*}
	 g(0) = \thalf \norm{ B x }_2^2 . 
\end{equation*}
Hence, it suffices to show that $ 0 $ minimizes $ g $
if and only if $ \norm{ B\tr\! B x }_\infty \le 1 $.
Since $ g $ is convex, 
$ 0 $ minimizes $ g $ if and only if $ 0 \in \partial g( 0 ) $
where $ \partial g $ is the subdifferential of $ g $ given by
\begin{equation*}
	\partial g(v) = \sign(v) + B\tr\! B ( v - x )
\end{equation*}
where
$ \sign $ is the set-valued signum function,
\begin{equation*}
	\sign( t ) :=
	\begin{cases}	
		\{ 1 \}, & t > 0
		\\
		[-1, \ 1], \ & t = 0
		\\
		\{ -1 \}, & t < 0.
	\end{cases}
\end{equation*}
It follows that $ 0 $ minimizes $ g $ if and only if
\begin{alignat*}{2}
	0  \in \sign(0) - B\tr\! B x
	\quad
	\Leftrightarrow &&
	B\tr\! B x
	&
	\in  [-1, \ 1]^N
	\\
	\Leftrightarrow &&
	\bigl[ B\tr\! B x \bigr]_n & \in [-1, \ 1] \ \ \text{for $ n = 1, \dots, N $}
	\\
	\Leftrightarrow & \quad &
	\norm{ B\tr\! B x }_\infty & \le 1.
\end{alignat*}
Hence, the function $ S_B $ coincides with 
$ \thalf \norm{ B \adot }_2^2 $
on a subset of its domain. 
\end{proof}

\begin{prop}
Let $ B \in \RR^{ M \times N } $
and
set $ \alpha = \norm{ B }_2 $.
The generalized Huber function satisfies
\begin{align}
	S_B(x) 
	& \le S_{ \alpha I } ( x ),
	\quad
	\forall x \in \RR^N
	\\
	\label{eq:SBI}
	& =
	\sum_n s_\alpha( x_n).
\end{align}
\end{prop}
\begin{proof}
Using \eqref{eq:defSm}, we have 
\begin{align*}
	S_B(x)
	& =
	\min_{ v \in \RR^N }
	\bigl\{		\norm{ v }_1 + \thalf \norm{ B ( x - v) }_2^2		\bigr\}
	\\
	& \le
	\min_{ v \in \RR^N }
	\bigl\{		\norm{ v }_1 + \thalf \norm{ B }_2^2 \, \norm{ ( x - v) }_2^2		\bigr\}
	\\
	& =
	\min_{ v \in \RR^N }
	\bigl\{		\norm{ v }_1 + \thalf \alpha^2 \norm{ ( x - v) }_2^2		\bigr\}
	\\
	& =
	\min_{ v \in \RR^N }
	\bigl\{		\norm{ v }_1 + \thalf  \norm{ \alpha \, ( x - v) }_2^2		\bigr\}
	\\
	& = S_{ \alpha I }(x).
\end{align*}
From Proposition \ref{prop:BBD} we have \eqref{eq:SBI}.
\end{proof}

The Moreau envelope is well studied in convex analysis
\cite{Bauschke_2011}.
Hence, 
it is useful to express the generalized Huber function $ S_B $ in terms of 
a Moreau envelope,
so we can draw on results in convex analysis to
derive further properties of the generalized Huber function. 

\begin{lemma}
If $ B \in \RR^{ N \times N } $ is invertible, then
the generalized Huber function $ S_B $ can be expressed in terms of a Moreau envelope as
\begin{equation}
	\label{eq:SBME1}
	S_B = \big( \norm{ \adot }_1 \compose B\inv )\me \compose B.
\end{equation}

\end{lemma}

\begin{proof}
Using \eqref{eq:defS}, we have
\begin{align*}
	S_B
	& =
	\norm{ \adot }_1 \infconv \bigr( \thalf \norm{ \adot }_2^2 \compose B \bigr)
	\\
	& =
	\Bigl( \norm{ \adot }_1 \infconv \bigr( \thalf \norm{ \adot }_2^2 \compose B \bigr) \Bigr) \compose B\inv \compose B
	\\
	& =
	\Bigl( \bigl( \norm{ \adot }_1 \compose B\inv \bigr) \infconv \bigl( \thalf \norm{ \adot }_2^2 \bigr) \Bigr) \compose B
	\\
	& =
	\bigl( \norm{ \adot }_1 \compose B\inv \bigr)\me \compose B .
	\notag
	\qedhere
\end{align*}

\end{proof}

\begin{lemma}
\label{lemma:me}
If $ B \in \RR^{ M \times N } $ has full row-rank,
then the generalized Huber function $ S_B $ can be expressed in terms of a Moreau envelope as
\begin{equation}
	\label{eq:SBME2}
	S_B = \big( d \compose B\pinv )\me \compose B
\end{equation}
where $ d \colon \RR^N \to \RR $ is the convex distance function 
\begin{equation}
	\label{eq:defdist}
	d(x) = \min_{ w \in \nulls B } \norm{ x - w }_1
\end{equation}
which represents the distance from the point $ x \in \RR^N $ to the null space of $ B $ as measured by the $ \ell_1 $ norm.
\end{lemma}
\begin{proof}
Using \eqref{eq:defSm}, we have
\begin{align*}
	S_B(x)
	& =
	\min_{ v \in \RR^N }
	\bigl\{	 \norm{ v }_1 + \thalf \norm{ B ( x - v) }_2^2 	\bigr\}
	\\
	& = f( B x)
\end{align*}
where $ f \colon \RR^M \to \RR $ is given by
\begin{align*}
	f(z) & = 
	\min_{ v \in \RR^N }
	\bigl\{	\norm{ v }_1 + \thalf \norm{ z - B v }_2^2		\,	\bigr\}
	\\
	& =
	\min_{ u \in (\nulls B)^\perp }
	\;
	\min_{ w \in \nulls B }
	\bigl\{		\norm{ u + w }_1 + \thalf \norm{ z - B ( u + w )  }_2^2		\,	\bigr\}
	\\
	& = 
	\min_{ u \in (\nulls B)^\perp }
	\;
	\min_{ w \in \nulls B }
	\bigl\{		\norm{ u + w }_1 + \thalf \norm{ z - B u }_2^2		\,	\bigr\}
	\\
	& = 
	\min_{ u \in (\nulls B)^\perp }
	\bigl\{		d(u) + \thalf \norm{ z - B u }_2^2		\,	\bigr\}
\end{align*}
where $ d $ is the convex function given by \eqref{eq:defdist}.
The fact that $ d $ is convex follows from Proposition 8.26 of \cite{Bauschke_2011}
and Examples 3.16 and 3.17 of \cite{Boyd_2004}.
Since $ (\nulls B)^\perp = \ranges B\tr $, 
\begin{align*}
	f(z)
	& = 
	\min_{ u \in \ranges B\tr  }
	\bigl\{	d(u) + \thalf \norm{ z - B u }_2^2		\bigr\}
	\\
	& = 
	\min_{ v \in \RR^M  }
	\bigl\{	d( B\tr\! v ) + \thalf \norm{ z - B B\tr\! v }_2^2		\bigr\}
	\\
	& = 
	\min_{ v \in \RR^M  }
	\bigl\{	d( B\tr (B B\tr)\inv v ) + \thalf \norm{ z - B B\tr  (B B\tr)\inv  v }_2^2	\bigr\}
	\\
	& = 
	\min_{ v \in \RR^M  }
	\bigl\{	d( B\pinv v ) + \thalf \norm{ z - v }_2^2		\bigr\}
	\\
	& = 
	\bigl( d ( B\pinv \adot ) \bigr)\me (z).
\end{align*}
Hence,
$ S_B(x) = \bigl( d ( B\pinv \adot ) \bigr)\me (Bx) $ 
which completes the proof. 
\end{proof}

Note that \eqref{eq:SBME2} reduces to \eqref{eq:SBME1} when $ B $ is invertible.
(Suppose $ B $ is invertible. 
Then $ \nulls B = \{ 0 \} $;  hence $ d(x) = \norm{ x }_1 $ in \eqref{eq:defdist}.
Additionally, $ B\pinv = B\inv $.)

\begin{prop}
The generalized Huber function is differentiable.
\end{prop}
\begin{proof}
By Lemma \ref{lemma:me}, $ S_B $ is the composition of a Moreau envelope of a convex function
and a linear function.
Additionally, by Proposition \ref{prop:lsc}, $ S_B \in \Gz(\RR^N) $.
By Proposition 12.29 in \cite{Bauschke_2011}, it follows that $ S_B $ is differentiable.
\end{proof}

The following result regards the gradient of the generalized Huber function.
This result will be used in Sec.~\ref{sec:gmc} to show 
the generalized MC penalty defined therein 
constitutes a valid penalty.

\begin{lemma}
\label{gradbnd}
The gradient of the generalized Huber function $ S_B \colon \RR^N \to \RR $ satisfies
\begin{equation}
	\norm{ \nabla S_B(x) }_\infty \le 1
	\ \
	\text{for all} \ x \in \RR^N.
\end{equation}
\end{lemma}
\begin{proof}
Since $ S_B $ is convex and differentiable, we have
\begin{equation*}
	S_B(v) + \bigl[ \nabla S_B(v) \bigr]\tr ( x - v ) \le S_B(x),
	\quad
	\forall x \in \RR^N, \;
	\forall v \in \RR^N.
\end{equation*}
Using Proposition \ref{prop:SBltL1}, it follows that
\begin{equation*}
	S_B(v) + \bigl[ \nabla S_B(v) \bigr]\tr ( x - v ) \le \norm{ x }_1,
	\quad
	\forall x \in \RR^N, \;
	\forall v \in \RR^N.
\end{equation*}
Let $ x = (0,\ldots,0,t,0,\ldots,0) $ where $ t $ is in position $ n $.
It follows 
that
\begin{equation}
	\label{eq:suplsb2}
	c(v) + \bigl[ \nabla S_B(v) \bigr]_n \,  t  \le \abs{ t },
	\quad
	\forall t \in \RR, \;
	\forall v \in \RR^N
\end{equation}
where $ c(v) \in \RR $ does not depend on $ t $.
It follows from \eqref{eq:suplsb2} that
$ \abs{ [ \nabla S_B(v) ]_n } \le 1 $.
\end{proof}

The generalized Huber function can be evaluated by taking
the pointwise minimum of numerous simpler functions
(comprising quadratics, absolute values, and linear functions).
This generalizes the situation for the scalar Huber function,
which can be evaluated as the pointwise minimum of three functions,
as expressed in \eqref{eq:smmm} and illustrated in Fig.~\ref{fig:scalar_huber_min}.
Unfortunately, 
evaluating the generalized Huber function on $ \RR^N $
this way requires the evaluation of $ 3^N $ simpler functions,
which is not practical except for small $ N $.
In turn, the evaluation of the GMC penalty is also impractical. 
However, we do not need to explicitly evaluate these functions to utilize them
for sparse regularization, as shown in Sec.~\ref{sec:alg}.
For this paper, 
we compute these functions on $ \RR^2 $ only 
for the purpose of illustration (Figs.~\ref{fig:huber2_rank2_A} and \ref{fig:huber2_rank1_A}).

\section{Generalized MC Penalty}
\label{sec:gmc}

In this section,
we propose a multivariate generalization of the MC penalty \eqref{eq:defMC}.
The basic idea is to generalize \eqref{eq:defsMC}
using the $ \ell_1 $ norm and the generalized Huber function.

\begin{defn}
\label{def:gmcpen}
Let $ B \in \RR^{ M \times N } $.
We define the generalized MC (GMC) penalty function $ \psi_B \colon \RR^N \to \RR $ as
\begin{equation}
	\label{eq:defpsi}
	\psi_B( x ) := \norm{ x }_1 - S_B(x)
\end{equation}
where $ S_B $ is the generalized Huber function \eqref{eq:defS}.
\end{defn}

The GMC penalty reduces to a separable penalty when
$ B\tr \! B $ is diagonal.

\begin{prop}
\label{prop:sBpenD}
Let $ B \in \RR^{ M \times N } $.
If $ B\tr\! B $ is a diagonal matrix, 
then $ \psi_B $ is separable (additive), 
comprising a sum of scalar MC penalties.
Specifically,
\begin{equation}
	\nonumber
	B\tr\! B = \diag( \alpha_1^2, \dots, \alpha_N^2 )
	\ \ \implies \ \
	\psi_B(x) = \sum_n \phi_{\alpha_n} ( x_n )
\end{equation}
where $ \phi_b $ is the scaled MC penalty \eqref{eq:defsMC}.
If $ B\tr \! B = 0 $, then $ \psi_B(x) = \norm{ x }_1 $.
\end{prop}
\begin{proof}
If
$ B\tr\! B = \diag( \alpha_1^2, \dots, \alpha_N^2 ) $,
then by Proposition \ref{prop:BBD} we have
\begin{align*}
	\psi_B(x)
	& = \norm{ x }_1 - \sum_n s_{ \alpha_n } ( x_n )
	\\
	& = \sum_n \, \abs{ x_n } - s_{ \alpha_n } ( x_n )
\end{align*}
which proves the result in light of definition \eqref{eq:defsMC}.
\end{proof}

\begin{figure}[t]
	\begin{minipage}[t]{0.45\textwidth}
	\centering
	\includegraphics{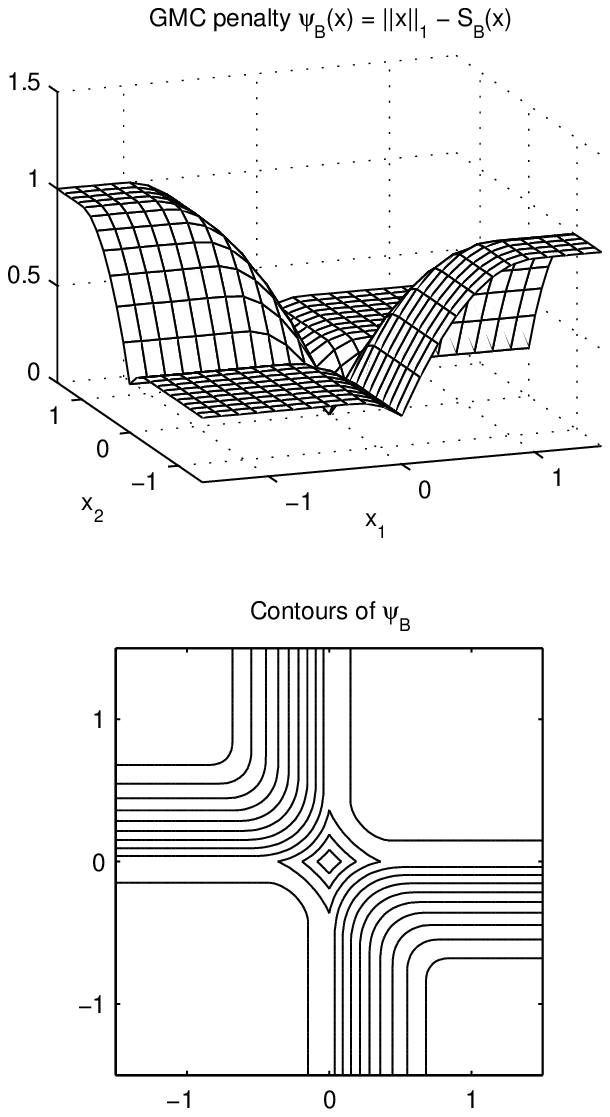}
	\caption{
		The GMC penalty 
		for the matrix $ B $ in \eqref{eq:BA}.
	}
	\label{fig:huber2_rank2_B}
	\end{minipage}
	\hfill
	\begin{minipage}[t]{0.45\textwidth}
	\centering
	\includegraphics{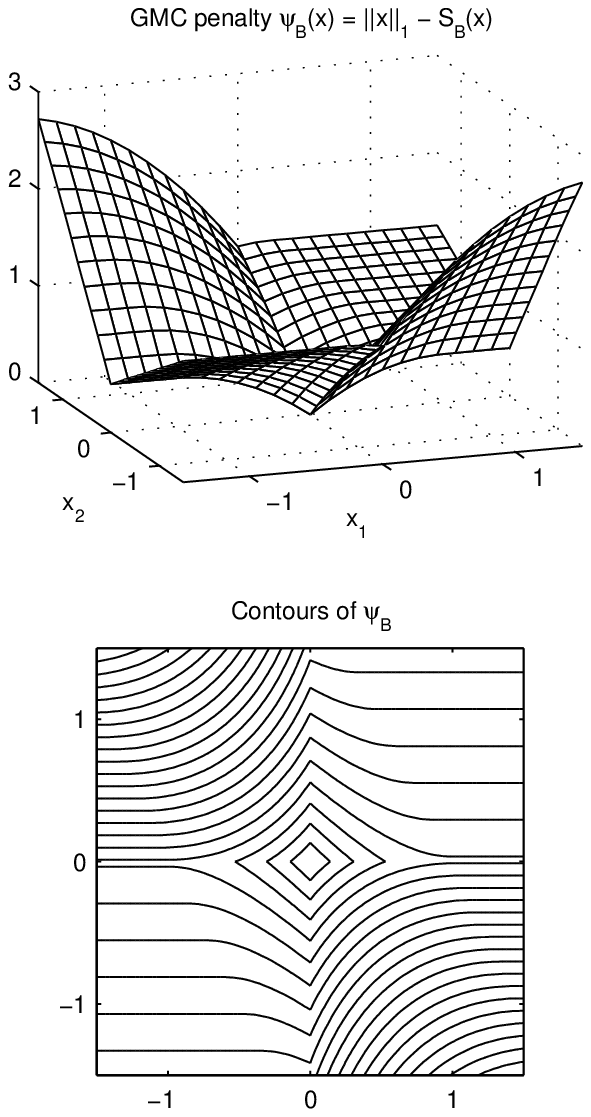}
	\caption{
		The GMC penalty 
		for the matrix $ B $ in \eqref{eq:BB}.
	}
	\label{fig:huber2_rank1_B}
	\end{minipage}
\end{figure}

The most interesting case (the case that motivates the GMC penalty)
is the case where $ B\tr \! B $ is a non-diagonal matrix.
If $ B\tr \! B $ is non-diagonal, 
then 
the GMC penalty is non-separable. 

\begin{example}

For the matrices $ B $ given in 
\eqref{eq:BA}
and 
\eqref{eq:BB},
the GMC penalty is illustrated
in Fig.~\ref{fig:huber2_rank2_B} and Fig.~\ref{fig:huber2_rank1_B},
respectively.

\end{example}

The following corollaries follow directly from Propositions \ref{prop:SBltL1} and \ref{prop:sbnz}.
 
\begin{corr}
\label{corr:MCbnds}
The generalized MC penalty satisfies
\begin{equation}
	0 \le
	\psi_B(x) \le \norm{ x }_1
	\quad
	\text{for all $ x \in \RR^N $.} 
\end{equation}
\end{corr}

\begin{corr}
\label{corr:mcnz}
Given $ B \in \RR^{ M \times N } $,
the generalized MC penalty
satisfies
\begin{equation}
	\psi_B(x) = \norm{ x }_1 - \thalf \norm{ B x }_2^2
	\ \
	\text{for all} \ \ 
	\norm{ B\tr\! B x }_\infty \le 1.
\end{equation}
\end{corr}
The corollaries imply
that around zero the generalized MC penalty approximates the $ \ell_1 $ norm (from below),
i.e., $ \psi_B(x) \approx \norm{ x }_1 $ for $ x \approx 0 $.

The generalized MC penalty has a basic property expected of a regularization function;
namely, that large values are penalized more than (or the same as) small values. 
Specifically, 
if $ v , x \in \RR^N $ with
$ \abs{ v_i } \ge \abs{ x_i } $
and $ \sign{ v_i } = \sign{ x_i } $
for $ i = 1,\ldots,N $, then
$ \psi_B(v) \ge \psi_B(x) $.
That is, in any given quadrant, the function $ \psi_B(x) $ 
is 
a non-decreasing
function in each $ \abs{ x_i } $.
This is formalized in the following proposition,
and illustrated in
Figs.~\ref{fig:huber2_rank2_B} and \ref{fig:huber2_rank1_B}.
Basically, the gradient of $ \psi_B $ points away 
from the origin.

\begin{prop}
\label{prop:increasing}
Let $ x \in \RR^N $ with $ x_i \neq 0 $.
The generalized MC penalty $ \psi_B $ has the property that 
$  [ \nabla \psi_B (x) ]_i $ either has the same sign as $ x_i $ or is equal to zero.
\end{prop}

\begin{proof}
Let $ x \in \RR^N $ with $ x_i \neq 0 $.
Then, from the definition of the MC penalty,
\begin{equation*}
	\frac{ \partial \psi_B }{ \partial {x_i} }( x ) =
	\sign(x_i) - 
	\frac{ \partial S_B }{ \partial {x_i} }( x ).
\end{equation*}
From Lemma \ref{gradbnd},
$ \abs{ \partial S_B(x) / \partial x_i } \le 1 $.
Hence 
$ \partial \psi_B(x) / \partial x_i \ge 0 $ when $ x_i > 0 $,
and
$ \partial \psi_B(x) / \partial x_i \le 0 $ when $ x_i < 0 $.
\end{proof}
A penalty function not satisfying Proposition \ref{prop:increasing}
would not be considered an effective sparsity-inducing regularizer.

\section{Sparse Regularization}
\label{sec:spreg}

In this section, we consider how to set the GMC penalty
to maintain the convexity of the regularized least square cost function.
To that end, 
the condition \eqref{eq:BltA} below generalizes the scalar convexity condition \eqref{eq:scc}.

\begin{theorem}
\label{thm:cc}
Let
$ y \in \RR^M $,
$ A \in \RR^{M \times N} $,
and
$ \lam > 0 $.
 Define $ F \colon \RR^N \to \RR $ as
\begin{equation}
	\label{eq:defFx}
	F(x) = \half \norm{ y - A x }_2^2 + \lam \, \psi_B(x)
\end{equation}
where
$ \psi_B \colon \RR^N \to \RR $ is
the generalized MC penalty \eqref{eq:defpsi}.
If 
\begin{equation}
	\label{eq:BltA}
	B\tr\! B \mle  \frac{ 1 }{ \lam } A\tr\! A 
\end{equation}
then $ F $ is a convex function. 

\end{theorem}

\begin{proof}
Write $ F $ as
\begin{align*}
	F(x) 
	& = 
	\thalf \norm{ y - A x }_2^2 + \lam \, \bigl(  \norm{ x }_1 - S_B(x) \bigr)
	\\
	& = 
	\thalf \norm{ y - A x }_2^2 + \lam \, \norm{ x }_1 
	- 
	\min_{ v \in \RR^N }
	\big\{
		\lam \norm{ v }_1 + \tfrac{ \lam }{ 2 } \norm{ B ( x - v) }_2^2
	\big\}
	\\
	& = 
	\max_{ v \in \RR^N }
	\big\{
		\thalf \norm{ y - A x }_2^2 + \lam \, \norm{ x }_1
		- 
		\lam \norm{ v }_1
		- \tfrac{ \lam }{ 2 } \norm{ B ( x - v) }_2^2
	\big\}
	\\
	\nonumber
	& = 
	\max_{ v \in \RR^N }
	\big\{
		\thalf x\tr\! \bigl( A\tr\! A - \lam B\tr\! B \bigr) x  + \lam \norm{ x }_1 
		+ g(x, v)
	\big\}
	\\
	& = 
	\thalf x\tr\! \bigl( A\tr\! A - \lam B\tr\! B \bigr) x  + \lam \norm{ x }_1 
	+
	\max_{ v \in \RR^N }
	g(x, v)
\end{align*}
where $ g $ is affine in $ x $.
The last term is convex as it is the pointwise maximum of a set of convex functions
(Proposition 8.14 in \cite{Bauschke_2011}).
Hence, $ F $ is convex if $ A\tr\! A - \lam B\tr\! B $ is positive semidefinite.
\end{proof}

The convexity condition \eqref{eq:BltA} is easily satisfied. 
Given $ A $, we may simply set
\begin{equation}
	\label{eq:setB}
	B = \sqrt{ \gamma / \lam \, } A,
	\quad
	0 \le \gamma \le 1. 
\end{equation}
Then
$ B\tr\! B = ( \gamma / \lam ) A\tr\! A $
which satisfies \eqref{eq:BltA} when $ \gamma \le 1 $.
The parameter $ \gamma $ controls the non-convexity of the penalty $ \psi_B $.
If $ \gamma = 0 $, then $ B = 0 $ and the penalty reduces to the $ \ell_1 $ norm.
If $ \gamma = 1 $, then \eqref{eq:BltA} is satisfied with equality
and the penalty is `maximally' non-convex. 
In practice, we use a nominal range of $ 0.5 \le \gamma \le 0.8 $.

\bigskip

When $ A\tr \! A $ is diagonal, 
the proposed methodology
reduces to element-wise firm thresholding. 

\begin{prop}
Let
$ y \in \RR^M $,
$ A \in \RR^{M \times N} $,
and
$ \lam > 0 $.
If $ A\tr \! A $ is diagonal with positive diagonal entries
and
 $ B $ is given by \eqref{eq:setB},
then the minimizer of the 
cost function $ F $ in \eqref{eq:defFx}
is given by element-wise firm thresholding.
Specifically, if
\begin{equation}
	A\tr \! A = \diag( \alpha_1^2, \dots, \alpha_N^2 ),
\end{equation}
then
\begin{equation}
	\label{eq:xsopt}
	x\opt_n = \firm( [A\tr\! y]_n / \alpha_n^2 ; \lam / \alpha_n^2 , \lam / (\gamma \alpha_n^2 ) )
\end{equation}
when $ 0 < \gamma \le 1 $,
and
\begin{equation}
	x\opt_n = \soft( [A\tr y]_n / \alpha_n^2 ; \lam / \alpha_n^2) 
\end{equation}
when $ \gamma = 0 $.

\end{prop}

\begin{proof}
If $ A\tr \! A = \diag( \alpha_1^2, \dots, \alpha_N^2 ) $, then
\begin{align*}
	\thalf \norm{ y - A x }_2^2
	& = \thalf y\tr \! y
	- x\tr \! A\tr y + \thalf x\tr \! A\tr \! A x
	\\
	& =
	\thalf y\tr \! y + \sum_n \bigr( -x_n [A\tr \! y]_n + \thalf \alpha_n^2 x_n^2  \bigl)
	\\
	& =
	\sum_n \thalf \bigl(  [ A\tr\! y]_n / \alpha_n  - \alpha_n x_n \bigr)^2 + C
\end{align*}
where $ C $ does not depend on $ x $.
If $ B $ is given by \eqref{eq:setB}, then
\begin{equation*}
	B\tr \! B = ( \gamma / \lam ) \, \diag( \alpha_1^2, \dots, \alpha_N^2 ).
\end{equation*}
Using Proposition \ref{prop:sBpenD}, we have
\begin{equation*}
	\psi_B(x) = \sum_n \phi_{\alpha_n \sqrt{ \gamma / \lam } } ( x_n ).
\end{equation*}
Hence, $ F $ in \eqref{eq:defFx} is given by
\begin{equation*}
	F(x) = 
	\sum_n
	\left[
	 \thalf \bigl(  [ A\tr\! y]_n / \alpha_n  - \alpha_n x_n \bigr)^2 
	 	+
		\lam
		 \phi_{\alpha_n \sqrt{ \gamma / \lam } } ( x_n )
	 \right] + C
\end{equation*}
and so \eqref{eq:xsopt} follows from \eqref{eq:sfirm}.
\end{proof}

\section{Optimization Algorithm}
\label{sec:alg}

Even though the GMC penalty does not have a simple explicit formula, 
a global minimizer of the sparse-regularized cost function \eqref{eq:defFx}
can be readily calculated using proximal algorithms.
It is not necessary to explicitly evaluate the GMC penalty or its gradient.

To use proximal algorithms to minimize the cost function  $ F $ in \eqref{eq:defFx} 
when $ B $ satisfies \eqref{eq:BltA}, 
we rewrite it as a saddle-point problem:
\begin{equation}
	( x\opt, \, v\opt )
	=
	\arg
	\min_{ x \in \RR^N }
	\max_{ v \in \RR^N }
	F(x, v) 
\end{equation}
where
\begin{equation}
	F(x, v) = 
		\half \norm{ y - A x }_2^2 + \lam \, \norm{ x }_1 - 
		\lam \norm{ v }_1
		- \frac{ \lam }{ 2 } \norm{ B ( x - v) }_2^2
\end{equation}

If we use \eqref{eq:setB} with $ 0 \le \gamma \le 1 $,
then the saddle function is given by
\begin{equation}
	\label{eq:defFxv}
	F(x, v) = 
		\half \norm{ y - A x }_2^2 + \lam \, \norm{ x }_1 - 
		\lam \norm{ v }_1
		- \frac{ \gamma }{ 2 } \norm{ A ( x - v) }_2^2.
\end{equation}

These saddle-point problems are instances of monotone inclusion problems.
Hence, the solution can be obtained using the forward-backward (FB) algorithm for such a problems;
see Theorem 25.8 of Ref.~\cite{Bauschke_2011}.
The FB algorithm involves only simple computational steps (soft-thresholding and the operators $ A $ and $ A\tr $).

\begin{prop}
\label{prop:alg}
Let
$ \lam > 0 $ and  $ 0 \le \gamma < 1 $.
Let 
$ y \in \RR^N $ and $ A \in \RR^{M \times N} $.
Then a saddle-point $ ( x\opt, \, v\opt ) $ of $ F $ in  \eqref{eq:defFxv} can be obtained by the iterative algorithm:
\begin{align*}
	&
	\text{Set }
	\rho = \max\{ 1, \, \gamma/(1-\gamma) \} \, \norm{ A\tr\! A }_2
	\\
	&
	\text{Set } 	\mu : 0 < \mu < 2 / \rho
	\\
	& \text{For $ i = 0, 1, 2, \dots $}
	\\
	& \quad
	w\iter{i} = 
	x\iter{i} - \mu A\tr\! \bigl( A( x\iter{i} + \gamma( v\iter{i} - x\iter{i} )) - y \bigr)
	\\
	& \quad
	u\iter{i} = v\iter{i} - \mu \gamma  A\tr\! A( v\iter{i} - x\iter{i} ) 
	\\
	& \quad
	x\iter{i+1} = \soft( w\iter{i}, \, \mu \lam)
	\\
	& \quad
	v\iter{ i + 1} = \soft( u\iter{i}, \, \mu \lam)
	\\
	& \text{end}
\end{align*}
where $ i $ is the iteration counter.
\end{prop}

\begin{proof}
The point $ ( x\opt , v\opt ) $ is a saddle-point of $ F $ if
$  0 \in \partial F( x\opt, v\opt ) $
where $ \partial F $ is the subdifferential of $ F $.
From \eqref{eq:defFxv}, 
we have
\begin{align*}
	\partial_x F(x, v) 
	& = 
	A\tr\! ( A x - y ) - \gamma A\tr\! A ( x - v ) + \lam \sign( x )
	\\
	\partial_v F(x, v) 
	& = 
	\gamma A\tr\! A ( x - v ) - \lam \sign( v ).
\end{align*}
Hence, $ 0 \in \partial F $ if $ 0 \in P(x, v) + Q(x, v) $ 
where
\begin{align*}
	P( x, v ) & = 
	\begin{bmatrix}
		(1 - \gamma) A\tr\! A 
		&
		\gamma A\tr\! A
		\\
		- \gamma A\tr\! A
		&
		\gamma A\tr\! A
	\end{bmatrix}
	\begin{bmatrix}
		x
		\\
		v
	\end{bmatrix}
	-
	\begin{bmatrix}
		A\tr y
		\\
		0
	\end{bmatrix}
	\\
	Q( x, v ) & = 
	\begin{bmatrix}
		\lam \sign( x )
		\\
		\lam \sign( v )
	\end{bmatrix}.
\end{align*}
Finding $ (x, v) $ such that $ 0 \in P(x, v) + Q(x, v) $ 
is the problem of constructing a zero of a sum of operators. 
The operators $ P $ and $ Q $ are maximally monotone
and $ P $ is single-valued 
and $ \beta $-cocoercive with $ \beta > 0 $;
hence, 
the forward-backward algorithm (Theorem 25.8 in \cite{Bauschke_2011}) can be used.
In the current notation, the forward-backward algorithm is
\begin{align*}
	&
	\begin{bmatrix}
		w\iter{i} \\ u\iter{i}
	\end{bmatrix}
	=
	\begin{bmatrix}
		x\iter{i} \\ v\iter{i}
	\end{bmatrix}
	 - \mu P ( x\iter{i}, \, v\iter{i} )
	\\
	& 
	\begin{bmatrix}
		x\iter{ i + 1 } \\ v\iter{ i + 1 }
	\end{bmatrix}
	= J_{ \mu Q } ( w\iter{i}, \, u\iter{i} )
\end{align*}
where $ J_{ Q } = (I + Q)\inv $ is the \emph{resolvent} of $ Q $.
The resolvent of the $ \sign $ function is soft thresholding.
The constant $ \mu $ should be chosen
$ 0 < \mu < 2 \beta $ where $ P $ is $ \beta $-cocoercive (Definition 4.4 in \cite{Bauschke_2011}),
i.e., $ \beta P $ is firmly non-expansive. 
We now address the value $ \beta $.
By Corollary 4.3(v) in \cite{Bauschke_2011}, this condition is equivalent to
\begin{equation}
	\label{eq:Pcond}
	\thalf P + \thalf P\tr - \beta P\tr P \mge 0.
\end{equation}
We may write $ P $ using a Kronecker product,
\begin{equation*}
	P = 
	\begin{bmatrix}
		1-\gamma 	& 	\gamma
		\\
		-\gamma		& 	\gamma
	\end{bmatrix}
	\otimes
	A\tr\! A.
\end{equation*}
Then we have
\begin{align*}
	& \thalf P + \thalf P\tr - \beta P\tr P 
	\\
	&
	=
	\begin{bmatrix}
		1 - \gamma 	& 	0
		\\
		0		& 		\gamma
	\end{bmatrix}
	\otimes
	A\tr\! A
	-
	\beta
	\begin{bmatrix}
		1-\gamma 	& 	-\gamma
		\\
		\gamma		& 	\gamma
	\end{bmatrix}
	\begin{bmatrix}
		1-\gamma 	& 	\gamma
		\\
		-\gamma		& 	\gamma
	\end{bmatrix}
	\otimes
	(A\tr\! A)^2
	\\
	& 
	=
	\left(
	\left(
	\begin{bmatrix}
		1 - \gamma 	& 	0
		\\
		0		& 		\gamma
	\end{bmatrix}
	-
	\beta_1
	\begin{bmatrix}
		1-\gamma 	& 	-\gamma
		\\
		\gamma		& 	\gamma
	\end{bmatrix}
	\begin{bmatrix}
		1-\gamma 	& 	\gamma
		\\
		-\gamma		& 	\gamma
	\end{bmatrix}
	\right)
	\otimes
	I_N
	\right)
	\left(
	I_2
	\otimes
	\bigl(
	I_N
	-
	\beta_2
	A\tr\! A
	\bigr)
	\right)
	\left(
	I_2 
	\otimes
	A\tr\! A
	\right)
\end{align*}
where $ \beta = \beta_1 \beta_2 $.
Hence, \eqref{eq:Pcond} is satisfied if
\begin{equation*}
	\begin{bmatrix}
		1 - \gamma 	& 	0
		\\
		0		& 		\gamma
	\end{bmatrix}
	-
	\beta_1
	\begin{bmatrix}
		1-\gamma 	& 	-\gamma
		\\
		\gamma		& 	\gamma
	\end{bmatrix}
	\begin{bmatrix}
		1-\gamma 	& 	\gamma
		\\
		-\gamma		& 	\gamma
	\end{bmatrix}
	\mge 
	0
\end{equation*}
and
\begin{equation*}
	I_N
	-
	\beta_2
	A\tr\! A
	\mge 
	0.
\end{equation*}
These conditions are repsectively satisfied
if 
\begin{equation*}
	\beta_1 \leq  1 / \max\{ 1, \; \gamma/(1-\gamma) \} 
\end{equation*}
and
\begin{equation*}
	\beta_2 \leq  1 / \norm{ A\tr\! A}_2.
\end{equation*}
The FB algorithm requires that $ P $ be $ \beta $-cocoercive with $ \beta > 0 $;
hence, $ \gamma = 1 $ is precluded. 
\end{proof}

If $ \gamma = 0 $ in Proposition \ref{prop:alg}, then the algorithm 
reduces to the classic iterative shrinkage/thresholding algorithm (ISTA) \cite{Daubechies_2004, Fig_2003_TIP}.

The Douglas-Rachford algorithm (Theorem 25.6 in \cite{Bauschke_2011})
may also be used to find a saddle-point of $ F $ in  \eqref{eq:defFxv}.

\section{Numerical Examples}
\label{sec:examples}

\subsection{Denoising using frequency-domain sparsity}

This example illustrates the use of the GMC penalty for denoising \cite{Chen_1998_SIAM}.  
Specifically, we consider the estimation of the discrete-time signal
\begin{equation*}
	g(m) = 2 \cos( 2 \pi f_1 m ) + \sin( 2 \pi f_2 m ),
	\qquad
	m = 0, \dots, M-1
\end{equation*}
of length $ M = 100 $ with frequencies $ f_1 = 0.1 $ and $ f_2 = 0.22 $.
This signal is sparse in the frequency domain, 
so we model the signal as
$ g = A x $ where $ A $ is an over-sampled inverse discrete Fourier transform
and $ x \in \CC^N $ is a sparse vector of Fourier coefficients with $ N \ge M $.
Specifically, 
we define the matrix $ A \in \CC^{M \times N} $
as
\begin{equation}
	A_{m,n} = \bigl(1/{\sqrt{N}} \, \bigr) \exp(\myJ (2\pi/N) mn),
	\qquad
	m = 0,\dots,M-1,
	\
	n = 0,\dots,N-1
\end{equation}
with $ N = 256 $.
The columns of $ A $ form a normalized tight frame, i.e., $ A A\ct = I $
where $ A\ct $ is the complex conjugate transpose of $ A $.
For the denoising experiment, we corrupt the signal with additive white Gaussian noise (AWGN) with
standard deviation $ \sigma = 1.0 $,
as illustrated in Fig.~\ref{fig:example_dft}.

\begin{figure}
	\centering
	\includegraphics{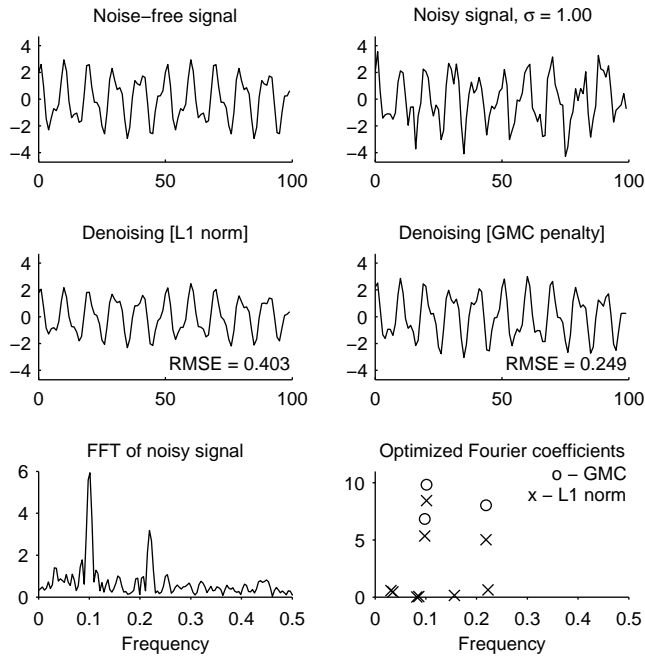}
	\caption{
		Denoising using the $ \ell_1 $ norm and the proposed GMC penalty.
		The plot of optimized coefficients shows only the non-zero values.
	}
	\label{fig:example_dft}
\end{figure}

In addition to the $ \ell_1 $ norm and proposed GMC penalty, 
we use several other methods:
debiasing the $ \ell_1 $ norm solution \cite{Figueiredo_2007_GPSR}, 
iterative p-shrinkage (IPS) \cite{Voronin_2013_ICASSP, Woodworth_2016_InvProb},
and
multivariate sparse regularization (MUSR) \cite{Selesnick_2017_TSP_MUSR}.
Debiasing the $ \ell_1 $ norm solution is a two-step approach 
where the $ \ell_1 $-norm solution is used to estimate the support,
then the identified non-zero values are re-estimated by un-regularized least squares.
The IPS algorithm is a type of iterative thresholding algorithm 
that performed particularly well in a detailed comparison of several algorithms \cite{Selesnick_2016_TSP_BISR}.
MUSR regularization is a precursor of the GMC penalty, 
i.e., a  non-separable non-convex penalty designed to maintain cost function convexity,
but with a simpler functional form. 

In this denoising experiment, we use 20 realizations of the noise. 
Each method calls for a regularization parameter $ \lam $ to be set.
We vary $ \lam $ from 0.5 to 3.5 (with increment 0.25) and evaluate the RMSE
for each method, for each $ \lam $, and for each realization.
For the GMC method we must also specify the matrix $ B $,
which we set using \eqref{eq:setB} with $ \gamma = 0.8 $.
Since $ B\ct \! B $ is not diagonal, the GMC penalty is non-separable. 
The average RMSE as a function of $ \lam $ 
for each method is shown in Fig.~\ref{fig:example_dft_rmse}.

The GMC compares favorably with the other methods, achieving the minimum average RMSE.
The next best-performing method is debiasing of the $ \ell_1 $-norm solution,
which performs almost as well as GMC.
Note that this debiasing method does not minimize an initially prescribed cost function,
in contrast to the other methods. 
The IPS algorithm aims to minimize a (non-convex) cost function.

Figure~\ref{fig:example_dft} shows the $ \ell_1 $-norm and GMC solutions for a particular noise realization.
The solutions shown in this figure were obtained using
the value of $ \lam $ that minimizes the average RMSE
($ \lam = 1.0 $ and $ \lam = 2.0 $, respectively). 
Comparing the $ \ell_1 $  norm and GMC solutions,
we observe:
the GMC solution is more sparse in the frequency domain;
and 
the $ \ell_1 $ norm solution underestimates the coefficient amplitudes.

Neither increasing nor decreasing the regularization parameter~$ \lam $ helps the $ \ell_1 $-norm solution here.
A larger value of $ \lam $ makes the $ \ell_1 $-norm solution sparser, but reduces 
the coefficient amplitudes.
A smaller value of $ \lam $ increases the coefficient amplitudes of the $ \ell_1 $-norm solution,
but makes the solution less sparse and more noisy. 

Note that the purpose of this example is to 
compare the proposed GMC penalty with other sparse regularizers.
We are not advocating it for frequency estimation \emph{per se}.

\begin{figure}[t]
	\centering
	\includegraphics{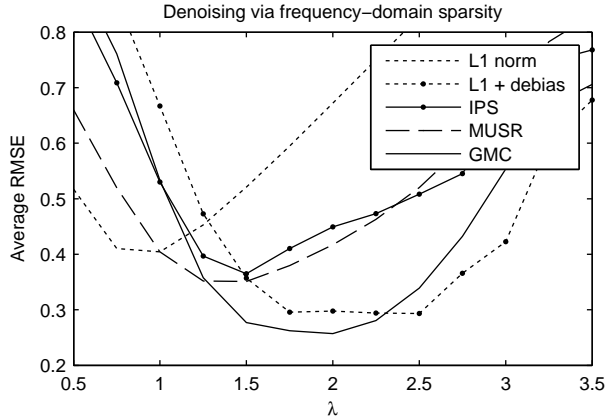}
	\medskip
	\caption{
		Average RMSE for three denoising methods.
	}
	\label{fig:example_dft_rmse}
\end{figure}

\subsection{Denoising using time-frequency sparsity}

This example considers the denoising of a bat echolocation pulse,
shown in Fig.~\ref{fig:example_bat} (sampling period of 7 microseconds).%
\footnote{The bat echolocation pulse data is curtesy of Curtis Condon, Ken White, and Al Feng of the Beckman Center at the University of Illinois. Available online at {http://dsp.rice.edu/software/bat-echolocation-chirp}.}
The bat pulse can be modeled as sparse in the time-frequency domain. 
We use a short-time Fourier transform (STFT) with 75\% overlapping segments
(the transform is four-times overcomplete). 
We implement the STFT as a normalized tight frame, i.e., $ A A\ct = I $.
The bat pulse and its spectrogram are illustrated in Fig.~\ref{fig:example_bat}.
For the denoising experiment, 
we contaminate the pulse with AWGN ($ \sigma = 0.05 $).

\begin{figure}
	\centering
		\includegraphics{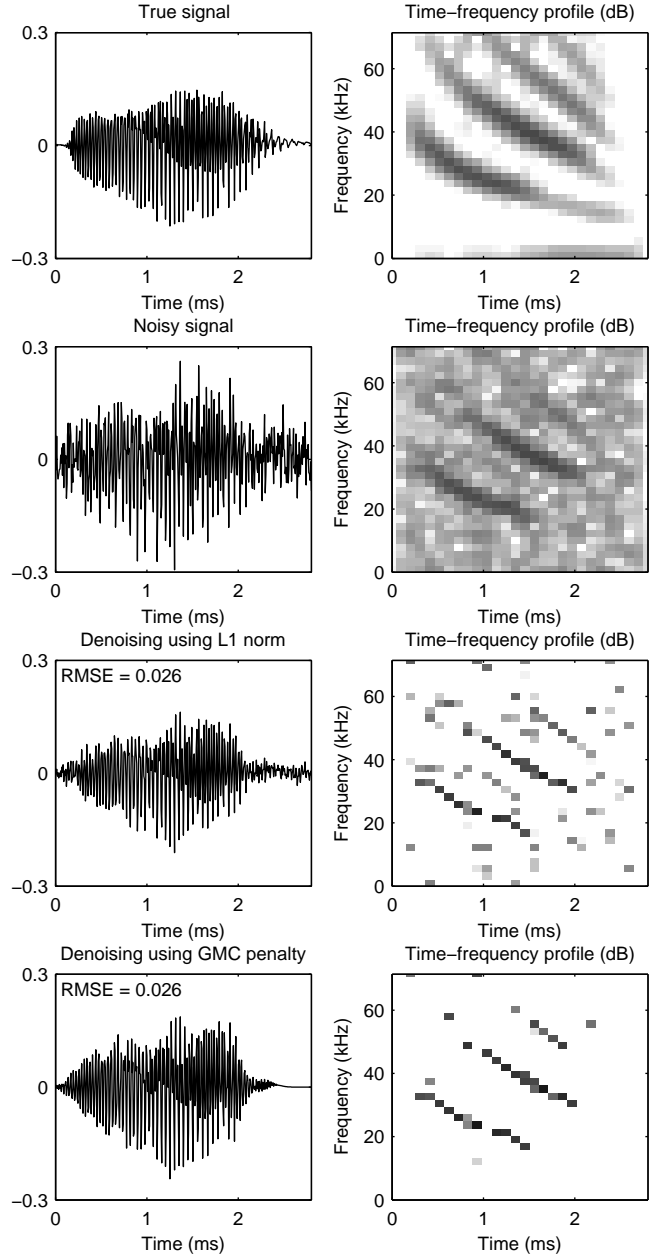}
	\caption{
	        Denoising a bat echolocation pulse using the $ \ell_1 $ norm and GMC penalty. 
	        The GMC penalty results in fewer extraneous noise artifacts
	        in the time-frequency representation.
	}
	\label{fig:example_bat}
\end{figure}

We perform denoising by estimating the STFT coefficients 
by minimizing the cost function $ F $ in \eqref{eq:defFx}
where $ A $ represents the inverse STFT operator.
We set $ \lam $ so as to minimize the root-mean-square error (RMSE).
This leads to the values  $ \lam = 0.030 $ and $ \lam = 0.51 $
for the $ \ell_1 $-norm and GMC penalties, respectively. 
For the GMC penalty, we set $ B $ as in \eqref{eq:setB} with $ \gamma = 0.7 $.
Since $ B\ct \! B $ is not diagonal, the GMC penalty is non-separable. 
We then estimate the bat pulse by
computing the inverse STFT of the optimized coefficients.
With $ \lam $ individually set for each method, 
the resulting RMSE is about the same (0.026).
The optimized STFT coefficients (time-frequency representation) for each solution is shown in Fig.~\ref{fig:example_bat}.
We observe that the GMC solution has substantially fewer
extraneous noise artifacts in the time-frequency representation,
compared to the $ \ell_1 $ norm solution.
(The time-frequency representations in Fig.~\ref{fig:example_bat} 
are shown in decibels with 0 dB being black and -50 dB being white.)

\subsection{Sparsity-assisted signal smoothing}

This example uses the GMC penalty for sparsity-assisted signal smoothing (SASS) \cite{Selesnick_2015_SASS, Selesnick_2017_ICASSP_SASS}.
The SASS method is suitable for the denoising of signals that are smooth for the exception of singularities. 
Here,
we use SASS to denoise the biosensor data illustrated in Fig.~\ref{fig:example_biosensor}(a),
which exhibits jump discontinuities. 
This data was acquired using a whispering gallery mode (WGM) sensor designed to detect nano-particles
with high sensitivity \cite{Dantham_2012, Arnold_2003_OL}.
Nano-particles show up as jump discontinuities in the data.

The SASS technique formulates the denoising problem as a sparse deconvolution problem.
The cost function to be minimized has the form \eqref{eq:defFx}.
The exact cost function, given by equation (42) in Ref. \cite{Selesnick_2017_ICASSP_SASS},
depends on a prescribed low-pass filter and the order of the singularities the signal is assumed to posses. 
For the biosensor data shown in Fig.~\ref{fig:example_biosensor},
the singularities are of order $ K = 1 $ since the first-order derivative of the signal exhibits impulses. 
In this example, we use a low-pass filter
of order $ d = 2 $ and cut-off frequency $ f_c = 0.01 $
(these parameter values designate a low-pass filter as described in \cite{Selesnick_2017_ICASSP_SASS}).
We set $ \lam = 32 $ and, for the GMC penalty, we set $ \gamma = 0.7 $.
Solving the SASS problem using the $ \ell_1 $ norm
and GMC penalty yields the denoised signals shown in Figs.~\ref{fig:example_biosensor}(b) and \ref{fig:example_biosensor}(c), respectively.
The amplitudes of the jump discontinuities are indicated in the figure.

\begin{figure}
	\centering
	\includegraphics{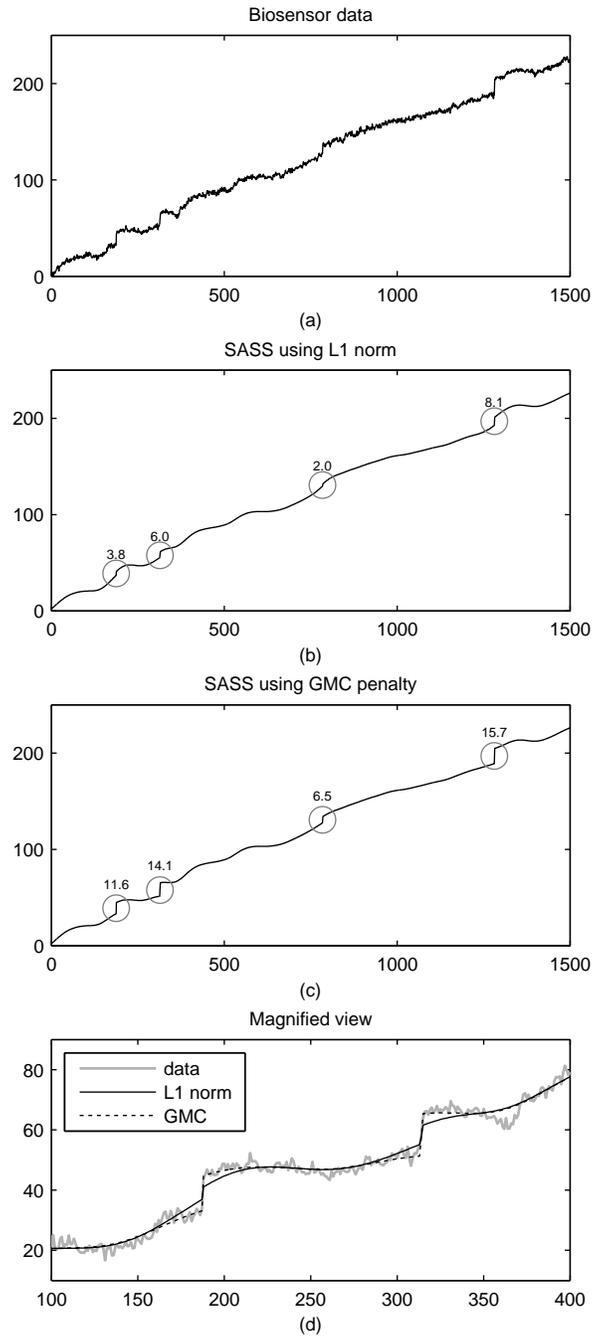}
	\caption{
		Sparsity-assisted signal smoothing (SASS) using
		$ \ell_1 $-norm and GMC regularization,
		as applied to biosensor data.
		The GMC method more accurately estimates jump discontinuities.
	}
	\label{fig:example_biosensor}
\end{figure}

It can be seen, especially in Fig.~\ref{fig:example_biosensor}(d), that the GMC solution estimates the jump discontinuities
more accurately than the $ \ell_1 $ norm solution.
The $ \ell_1 $ norm solution tends to underestimate the amplitudes of the jump discontinuities. 
To reduce this tendency, 
a smaller value of $ \lam $ could be used,
but that tends to produce false discontinuities (false detections).

\section{Conclusion}

In regards to  the sparse-regularized linear least squares problem,
this work bridges the convex (i.e., $ \ell_1 $ norm) and the non-convex (e.g., $ \ell_p $ norm with $ p < 1 $)
approaches, which are usually mutually exclusive and incompatible.
Specifically, 
this work formulates the sparse-regularized linear least squares problem
using a non-convex generalization of the $ \ell_1 $ norm 
that
preserves the convexity of the cost function to be minimized. 
The proposed method leads to optimization problems with no extraneous suboptimal local minima
and allows the leveraging of globally convergent, computationally efficient, scalable convex optimization algorithms. 
The advantage compared to $ \ell_1 $ norm regularization is
\ia\
more accurate estimation of high-amplitude components of sparse solutions
or
\ib\
a higher level of sparsity in a sparse approximation problem.
The sparse regularizer is expressed as the $ \ell_1 $ norm minus a 
smooth convex function defined via infimal convolution.
In the scalar case, the method reduces to firm thresholding
(a generalization of soft thresholding). 

Several extensions of this method are of interest.
For example, the idea may admit extension to more general convex regularizers 
such as
total variation \cite{ROF_1992}, 
nuclear norm \cite{Candes_2010_ProcIEEE},
mixed norms \cite{Kowalski_2009_SIVP},
composite regularizers \cite{Ahmad_2015_TCI, Afonso_2010_ICIP},
co-sparse regularization \cite{Nam_2011_acha},
and more generally, 
atomic norms \cite{Chandrasekaran_2012_FCM},
and
partly smooth regularizers \cite{Vaiter_2016_degrees}.
Another extension of interest is to problems where the
data fidelity term is not quadratic
(e.g., Poisson denoising \cite{Dupe_2009_TIP}).

\bibliographystyle{plain}

\end{document}